\documentclass[12pt,oneside]{article}
\usepackage[OT4]{fontenc}
\usepackage{a4}
\usepackage{amsmath}
\usepackage{amsfonts}
\usepackage{color}
\usepackage{amssymb}
\usepackage{tikz}
\usepackage{amsthm}
\usepackage{epsfig}
\usepackage{pstricks}
\usepackage{pst-node}
\usepackage[normalem]{ulem}
\usepackage{verbatim}
\usepackage{epic}\usepackage{eepic}
\definecolor{shadecolor}{RGB}{200,200,200}

\usepackage{verbatim}
\usepackage[utf8]{inputenc}
\newtheorem{thm}{Theorem}
\newtheorem{lem}[thm]{Lemma}

\newtheorem{wn}[thm]{Corollary}

\newtheorem{Problem}[thm]{Problem}
\newtheorem{example}{Example}

\theoremstyle{remark}
\def\cer{{\rm cer}}

\def\cP{\mathcal{P}}

\title{{\textsc{Graphs with equal domination and certified domination numbers}}}
\date{\today}
\begin{document}

\begin{center}
\noindent\textbf{\Large Graphs with equal domination\\[3mm] and certified domination numbers} \end{center}

\vspace{10mm}\begin{center}
\noindent\textsc{Magda Dettlaff, Magdalena Lema\'{n}ska}\\[1mm]
Gda\'nsk University of Technology, Gda\'nsk, Poland\\
{\small \texttt{\{mdettlaff,magda\}@mif.pg.gda.pl}}\\[3mm]
\textsc{Mateusz Miotk, Jerzy Topp, Rados\l{}aw Ziemann, Pawe\l{}~\.Zyli\'nski}\\[1mm]
University of Gda\'{n}sk, Gda\'nsk, Poland\\
{\small \texttt{\{m.miotk,j.topp,r.ziemann,p.zylinski\}@inf.ug.edu.pl}}
\end{center}

\vspace{1cm}
\begin{abstract}
\noindent A set $D$ of vertices of a graph $G$ is a {\it dominating set\/} of $G$ if every vertex in $V_G-D$ is adjacent to at least one vertex in $D$. The {\it domination number} ({\it upper domination number}, respectively) of a graph $G$, denoted by $\gamma(G)$ ($\Gamma(G)$, respectively), is the cardinality of a smallest
(largest minimal, respectively) dominating set of $G$. A subset $D\subseteq V_G$ is called a {\it certified dominating set\/} of $G$ if $D$ is a dominating set of $G$ and every vertex in $D$ has either zero or at least two neighbors in $V_G-D$. The cardinality of a~smallest (largest minimal, respectively) certified dominating set of $G$ is called the {\em certified} $(${\em upper certified}, respectively$)$  {\em domination number} of $G$ and is denoted by $\gamma_{\rm cer}(G)$ ($\Gamma_{\rm cer}(G)$, respectively). In this paper relations between domination, upper domination, certified domination and
upper certified domination numbers of a graph are studied.

\medskip
\noindent{\bf Keywords:}  Domination, certified domination, ${\cal P}$-corona of a graph \\
{\bf \AmS \; Subject Classification:} 05C69
\end{abstract}

\section{Notation and definitions}

We generally follow the notation and terminology of \cite{HHS98}. For a graph $G$, the set of vertices is denoted by $V_G$ and the edge set by $E_G$. For a vertex $v \in V_G$,
the {\em open neighborhood} $N_G(v)$ of $v$ is the set of all vertices adjacent to $v$, and $N_G[v]=N_G(v)\cup \{v\}$ is the {\em  closed neighborhood\/} of $v$.
The {\em open neighborhood} of a set $X\subseteq V_G$ is $N_G(X)= \bigcup_{v \in X} N_G(v)$, while the {\em closed neighborhood} of $X$ is the set $N_G[X]= N_G(X) \cup X$.  For $X \subseteq V_G$ and $v\in X$, the set $N_G[v]-N_G[X-\{v\}]$ is denoted by $P\!N_G[v,X]$ and called the {\it private neighborhood} of $v$ with respect to $X$. Every vertex belonging to $P\!N_G[v, X]$ is called a {\em private neighbor\/} of $v$ with respect to $X$. By $P\!N_G(v,X)$ we denote the set $N_G(v)-N_G[X-\{v\}]$ and call it the {\em open private neighborhood} (of $v$ with respect to $X$). The {\em degree} of a vertex $v$ in $G$ is the number $d_G(v)= |N_G(v)|$. The number $\min\{d_G(v)\colon v\in V_G\}$ is the minimum degree of $G$ and is denote by $\delta(G)$. A vertex of degree 0 is called an isolated vertex, while a~vertex of degree one in $G$ is called a {\em leaf\/} of $G$. If $v$ is a~leaf, then its only neighbor is called a {\em support} of $v$. A~support is called a {\em strong support} or {\em weak support} depending on whether or not it is adjacent to at least two leaves.
We use $L_G$, $S_G$, $S_G^1$ and $S_G^2$ to denote the set of all leaves, supports, weak supports and strong supports of $G$, respectively.

Given a graph $G$, we say that a subset $D\subseteq V_G$ is a {\it dominating set\/} of $G$ if every vertex belonging to $V_G-D$ is adjacent to at least one vertex in $D$. The {\it domination number} ({\it upper domination number}, respectively) of a graph $G$, denoted by $\gamma(G)$ ($\Gamma(G)$, respectively), is the cardinality of a smallest
(largest minimal, respectively) dominating set of $G$. A dominating (minimal dominating, respectively) set of $G$ of minimum (maximum, respectively) cardinality is called a~$\gamma$-set ($\Gamma$-set, respectively) of $G$. A subset $D\subseteq V_G$ is called a {\it certified dominating set\/} of $G$ if $D$ is a dominating set of $G$ and every vertex belonging to $D$ has either zero or at least two neighbors in $V_G-D$. The cardinality of a~smallest (largest minimal, respectively) certified dominating set of $G$ is called the {\em certified} $(${\em upper certified}, respectively$)$  {\em domination number} of $G$ and is denoted by $\gamma_{\rm cer}(G)$ ($\Gamma_{\rm cer}(G)$, respectively). A certified dominating (minimal certified dominating, respectively) set of $G$ of minimum (maximum, respectively) cardinality is called a~$\gamma_{\rm cer}$-set ($\Gamma_{\rm cer}$-set, respectively) of $G$. For example, it is easy to observe that for the most common graph families, we have $\gamma(K_n)= \gamma_{\rm cer}(K_n)= \Gamma(K_n)= \Gamma_{\rm cer}(K_n)=1$ if $n\not=2$, $\gamma(P_n) = \gamma_{\rm cer}(P_n)= \lceil n/3\rceil$ and $\Gamma_{\rm cer}(P_n)= \lfloor(n-1)/2\rfloor = \Gamma(P_n)-1$ if $n\ge 5$, $\gamma(C_n) = \gamma_{\rm cer}(C_n)= \lceil n/3\rceil$ and $\Gamma(C_n)= \Gamma_\cer(C_n)=\lfloor {n}/{2}\rfloor$ if $n\geq 3$, $\Gamma(K_{1,n})=n$ (if $n\ge 1$) and $\gamma(K_{1,n})=\gamma_\cer(K_{1,n})=\Gamma_\cer(K_{1,n})=1$ if $n\ge 2$,
$\gamma(K_{m,n})=\gamma_\cer(K_{m,n})=2$ and $\Gamma(K_{m,n})=\Gamma_\cer(K_{m,n})=n$ if $2\leq m\leq n$.

It is obvious that for any graph $G$ we have the inequalities $\gamma(G)\leq \gamma_{\cer}(G) \leq \Gamma_{\rm cer}(G)$. Comparing the parameters $\gamma(G)$, $\gamma_{\rm cer}(G)$, and $\Gamma_{\rm cer}(G)$ to $\Gamma(G)$, we may have the following three strings of inequalities:

\begin{equation}
\gamma(G)\leq \Gamma(G)\leq \gamma_{\cer}(G)\leq \Gamma_{\cer}(G); \label{nierowno1}\end{equation}

\begin{equation}  \gamma(G)\leq \gamma_{\cer}(G)\leq \Gamma(G) \leq \Gamma_{\cer}(G);\label{nierowno2}\end{equation}

\begin{equation} \gamma(G)\leq \gamma_{\cer}(G)\leq \Gamma_{\cer}(G)\leq \Gamma(G).\label{nierowno3}\end{equation}

\bigskip Each of the strings (1)--(3) is possible.
For example, the inequalities~(\ref{nierowno1}) hold for the graph $G$ in Figure~\ref{rys-graphsGFH}. In this case it is easy  to check that $\gamma(G)=5$, $\Gamma(G)=6$, $\gamma_{\rm cer}(G)=7$, $\Gamma_{\rm cer}(G)=8$, and the sets $\{v_1,v_2, v_3, v_5, v_6\}$, $\{v_5, v_8, v_{10},v_{11},v_{12}, v_{13}\}$, $\{v_1,v_2, v_3, v_5, v_6,v_{12},v_{13}\}$, $\{v_{1},v_{2},v_{3},$ $v_{5}, v_{8},v_{10}, v_{12},$ $v_{13}\}$ are  examples of $\gamma$-, $\Gamma$-, $\gamma_{\rm cer}$-, and $\Gamma_{\rm cer}$-sets of $G$, respectively. The graph $F$ in Figure~\ref{rys-graphsGFH} illustrates the inequalities~(\ref{nierowno2}). In this case one can check that
$\gamma(F)=3$, $\gamma_{\cer}(F)=4$, $\Gamma(F)=5$, $\Gamma_{\cer}(F)=6$, and the sets
$\{v_1,v_2, v_3\}$, $\{v_1,v_2, v_5, v_9\}$, $\{v_4, v_6, v_7, v_8, v_9\}$ and $\{v_1,v_2, v_4, v_6, v_7, v_9\}$ are  examples of $\gamma$-, $\gamma_{\rm cer}$-, $\Gamma$-, and $\Gamma_{\rm cer}$-sets of $F$, respectively. Finally, the inequalities (\ref{nierowno3}) hold for the graph $H$ in Figure~\ref{rys-graphsGFH}:  here $\gamma(H)=5$, $\gamma_{\cer}(H)=6$, $\Gamma_{\cer}(H)=7$, $\Gamma(H)=8$, and
$\{v_4,v_5,  v_7,v_{12}, v_{14}\}$, $\{v_4,v_5, v_6,v_7,v_{12},v_{14}\}$, $\{v_4,v_5,v_8, v_9,v_{10},v_{12},v_{14}\}$ and $\{v_1, v_2, v_3,$ $v_6, v_8, v_9, v_{10}, v_{13}\}$ are examples of $\gamma$-, $\gamma_{\rm cer}$-, $\Gamma_{\rm cer}$- and $\Gamma$-sets of $H$, respectively.

\begin{figure}[h!] \begin{center} \bigskip
{\epsfxsize=4.5in \epsffile{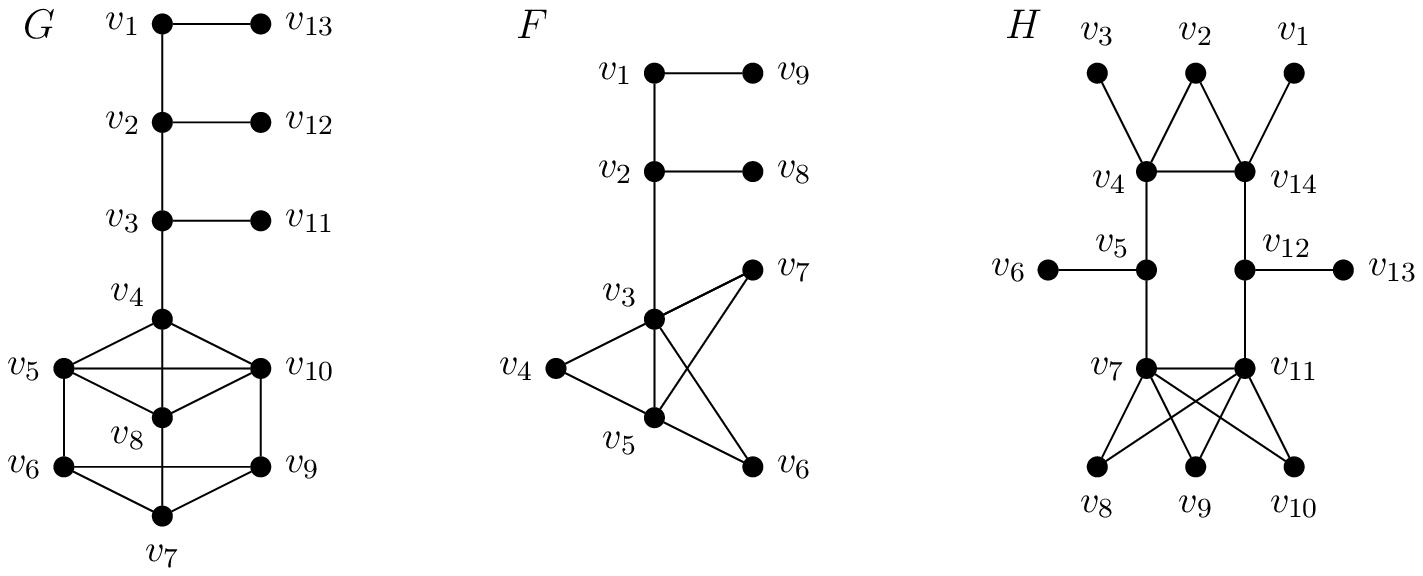}}

\vspace{-4mm}
\caption{Graphs $G$, $F$, and $H$.} \label{rys-graphsGFH}
\end{center}\end{figure}

Domination in graphs is one of the most fundamental and well-studied concepts in graph theory. The reader is referred to \cite{HHS98}, \cite{HHS98b} and \cite{HenningYeo} for more details on these important topics. The previously mentioned certified domination was introduced by Dettlaff et al. \cite{nasz} in order to describe some relations in social networks. In this paper we continue their studies of certified dominating sets and certified domination numbers of graphs. For different classes of graphs $G$ we establish  conditions for the equality of the domination number $\gamma(G)$ and the certified domination number $\gamma_{\rm cer}(G)$ of a graph $G$. Furthermore, we characterize all graphs $G$ for which  $\gamma(H)= \gamma_{\rm cer}(H)$ for each induced and connected subgraph $H\not=K_2$ of $G$. In addition, for a given graph $G$ we characterize all families ${\cal P}$ of partitions of vertex neighborhoods of $G$ for which $\gamma(G\circ {\cal P})= \gamma_{\rm cer}(G\circ {\cal P})$, where $G\circ {\cal P}$ is the ${\cal P}$-corona of $G$ defined in \cite{Dettlaff...Zylinski}.
The last part of the paper is concerned with main properties of the upper certified domination number $\Gamma_{\rm cer}(G)$ of $G$ and its relations to $\gamma_{\rm cer}(G)$ and $\Gamma(G)$. We conclude with some open problems.

\section{Graphs for which $\gamma_{\rm cer}=\gamma$}

In this section we study basic properties which guarantee equalities of domination and certified domination numbers. We begin with the following necessary and sufficient condition for the equality of domination and certified domination numbers of a graph.

\begin{thm} \label{obs_gamma_rowne} Let $G$ be a connected graph of order at least  three. Then  $\gamma(G)=\gamma_{\cer}(G)$ if and only if $G$ has a $\gamma$-set $D$ such that every vertex in $D$ has at least two neighbors in $V_G-D$. \end{thm}

\begin{proof} Assume that $\gamma(G)=\gamma_{\cer}(G)$. Let $D$ be a $\gamma_{\cer}$-set of $G$. Then the equality $\gamma(G)=\gamma_{\cer}(G)$ guarantees that  $D$ is also a $\gamma$-set of $G$. Now, let $v$ be a vertex in $D$. Since $v$ is not an isolated vertex, the set $N_G(v)\cap (V_G-D)$ is nonempty (as otherwise $D-\{v\}$ would be a~smaller dominating set of $G$). Consequently, since $D$ is a certified dominating set, we have $|N_G(v)\cap (V_G-D)| \ge 2$.

On the other hand, if $D$ is a $\gamma$-set of $G$ and $|N_G(v)\cap (V_G-D)| \ge 2$
for every $v\in D$, then $D$ is also a certified dominating set of $G$. Hence $\gamma_{\cer}(G)\leq |D|=\gamma(G)\leq \gamma_{\cer}(G)$ and therefore  $\gamma(G)=\gamma_{\cer}(G)$. \end{proof}

\begin{wn} \label{wn_gamma_rowne-1} Let $G$ be a connected graph of order at least three. If $G$ has an independent $\gamma$-set that contains no leaf of $G$, then $\gamma(G)=\gamma_{\cer}(G)$. \end{wn}

\begin{proof} Let $D$ be an independent $\gamma$-set of $G$ that contains no leaf of $G$ and let $v$ be any vertex in $D$. The independence of $D$ implies that $N_G(v)\subseteq V_G-D$. Now, since $v$ is neither a leaf nor an isolated vertex, we have $|N_G(v)|\ge 2$ and therefore $|N_G(v)\cap (V_G-D)|=|N_G(v)|\ge 2$. From this and from Theorem \ref{obs_gamma_rowne} it immediately follows that $\gamma(G)=\gamma_{\cer}(G)$. \end{proof}

It was already proved in \cite{nasz} that $\gamma(G)=\gamma_{\cer}(G)$
for all graphs $G$ without leaves. Here we present another proof of that result.

\begin{wn} \label{theorem1-1} If $G$ is a graph in which $\delta(G)\ge 2$, then: \begin{enumerate} \item[\rm (1)] $G$ has a $\gamma$-set $D$ such that every vertex in $D$ has at least two neighbors in $V_G-D$; \item[\rm (2)] $\gamma(G)=\gamma_{\cer}(G)$. \end{enumerate} \end{wn}

\noindent \begin{proof} If $X$ is a $\gamma$-set of $G$, then by $q(X)$ we denote the set $\{x\in X\colon |N_G(x) \cap (V_G - X)| \le 1\}$. It remains to show that $q(X)= \emptyset$ for some $\gamma$-set $X$ of $G$. Let $D$ be a $\gamma$-set of $G$.  If $q(D)=\emptyset$, then $D$ is the required set. Thus assume that $q(D) \neq \emptyset$. But now it suffices to show that there exists another $\gamma$-set $D'$ of $G$ for which $|q(D')|<|q(D)|$.

Let $v$ be any vertex in $q(D)$. Since $d_G(v)\ge \delta(G)\ge 2$, the set $N_G(v)$ cannot be a subset of $D$, as otherwise $D-\{v\}$ would be a dominating set of $G$. Consequently $|N_G(v)\cap (V_G-D)|=1$, say $N_G(v)\cap (V_G-D)=\{v'\}$.
Again, since $d_G(v)\ge 2$ and $|N_G(v)\cap (V_G-D)|=1$, $N_G(v)\cap D\not=
\emptyset$ and $v'$ is the only private neighbor of $v$ with respect to $D$.
Thus $N_G(v')-\{v\}$ is a non-empty subset of $V_G-D$ and $D'=(D-\{v\})\cup \{v'\}$ is a minimum dominating set of $G$ and the set $q(D')=\{x \in D'\colon |N_G(x)\cap (V_G-D')|
\le 1\}$ is smaller than $q(D)$ (as $q(D')=q(D)-\{v\}$). Therefore $|q(D')| < |q(D)|$
and this completes the proof of (1). The property (2) follows from (1) and Theorem   \ref{obs_gamma_rowne}. \end{proof}

It has been proved in \cite{GuntherHartnellMarkusRall} that if $D$ is a unique $\gamma$-set of a graph $G$, then every vertex in $D$ that is not an isolated vertex has at least two private neighbors other than itself, that is, in the set $V_G-D$.
From this and from Theorem \ref{obs_gamma_rowne} we have the following corollary.

\begin{wn}[\cite{nasz}] \label{gamma=gamma-unique} If a graph $G$ has a unique $\gamma$-set, then $\gamma(G)=\gamma_{\cer}(G)$. \end{wn}

The main properties of graphs having unique $\gamma$-sets have been studied in \cite{GuntherHartnellMarkusRall} and partialy in \cite{Fischermann}. It was also observed in \cite{GuntherHartnellMarkusRall} that if $D$ is a $\gamma$-set of a graph $G$ and $\gamma(G-x)>\gamma(G)$ for every $x\in D$, then $D$ is the unique $\gamma$-set of $G$. Thus, by Corollary \ref{gamma=gamma-unique}, we have the next corollary.

\begin{wn} \label{warunek-jedynosci} If a graph $G$ has a $\gamma$-set $D$ such that $\gamma(G-x)>\gamma(G)$ for every $x\in D$, then $\gamma(G)=\gamma_{\cer}(G)$. \end{wn}

The next theorem provides another sufficient condition for the equality of domination and certified domination numbers of a graph.

\begin{thm} \label{gamma_rowne} Let $G$ be a connected graph of order at least  three. Then  $\gamma(G)=\gamma_{\cer}(G)$ if $\gamma(G-v)\ge \gamma(G)$ for every vertex $v$ belonging to any $\gamma$-set of $G$. \end{thm}

\begin{proof} Similarly as in the proof of Corollary \ref{theorem1-1}, if $X$ is a $\gamma$-set of $G$, then by $q(X)$ we denote the set $\{x\in X\colon |N_G(x) \cap (V_G - X)| \le 1\}$. Assume that $\gamma(G-x)\ge \gamma(G)$ for every vertex $x$ belonging to any $\gamma$-set of $G$. To prove that $\gamma(G) = \gamma_{\cer}(G)$, by Theorem \ref{obs_gamma_rowne}, it remains to show that $q(X)= \emptyset$ for some $\gamma$-set $X$ of $G$. Let $D$ be a $\gamma$-set of $G$.  If $q(D)=\emptyset$, then $D$ is the required set. Thus assume that $q(D) \neq \emptyset$. Now it suffices to show that there exists another $\gamma$-set $D'$ of $G$ for which $|q(D')|<|q(D)|$.

Let $v$ be any vertex belonging to $q(D)$. Since $v$ is not an isolated vertex, the minimality of $D$ implies that $N_G(v) \cap (V_G - D)\not=\emptyset$ (as otherwise $D-\{v\}$ would be a smaller dominating set of $G$). Thus $|N_G(v) \cap (V_G - D)|=1$, and let $v'$ be the only vertex in $N_G(v) \cap (V_G - D)$. Then the set $D'=(D-\{v\})\cup \{v'\}$ is a~$\gamma$-set of $G$.

We now claim that $v'$ is a private neighbor of $v$ with respect to $D$. Suppose, contrary to our claim, that $v'\not\in P\!N_G[v,D]$. Then $v'\in N_G(D-\{v\})$ and either $N_G(v)\cap D\not=\emptyset$ or $N_G(v)\cap D=\emptyset$. The first case is impossible (as otherwise $D-\{v\}$ would be a smaller dominating set of $G$). Thus $N_G(v)\cap D=\emptyset$, but now $D-\{v\}$ is a dominating set of $G-v$ and therefore, $\gamma(G-v)\le |D-\{v\}|<|D|=\gamma(G)$, contrary to our assumption. This proves that $v'$ is a~private neighbor of $v$ with respect to $D$.

Next, since the private neighbor $v'$ of $v$ with respect to $D$ is the only neighbor of $v$ in $V_G-D$ and since $G$ is a connected graph of order at least  three, the sets $N_G(v)-\{v'\}$ and $N_G(v')-\{v\}$ are subsets of $D$ and $V_G-D$, respectively, and at least one of them is non-empty. If $N_G(v)-\{v'\}\not=\emptyset$, then $P\!N_G[v',D']=\{v'\}$ and $D'-\{v'\}$ is a dominating set of $G-v'$, and then $\gamma(G-v')\le |D'-\{v'\}|<|D'|=\gamma(G)$, contrary to our assumption. Thus assume
that $N_G(v)-\{v'\}=\emptyset$. Then $N_G(v')-\{v\}$ is a~non-empty subset of $V_G-D'$, and therefore $|N_G(v') \cap (V_G - D')|= |\{v\}|+|(N_G(v')-\{v\}) \cap (V_G - D')|=
1+|N_G(v')-\{v\}|\ge 2$. This implies that $v'\not\in q(D')$. Now, since $v\not\in q(D')$ (as $v\not\in D'$), neither $v'$ nor $v$ belongs to $q(D')$, and therefore  $q(D') \subseteq q(D)-\{v\}$. Consequently, we have $|q(D')|<|q(D)|$ and this completes the proof. \end{proof}

The fundamental relations between the classes of graphs considered in Corollaries \ref{theorem1-1}--\ref{warunek-jedynosci} and Theorems  \ref{obs_gamma_rowne} and \ref{gamma_rowne} are illustrated by examples in Figure \ref{rys-do-pracy}.\\

\begin{figure}[h!] \begin{center} \bigskip
\centerline{\epsfxsize=6.25in \epsffile{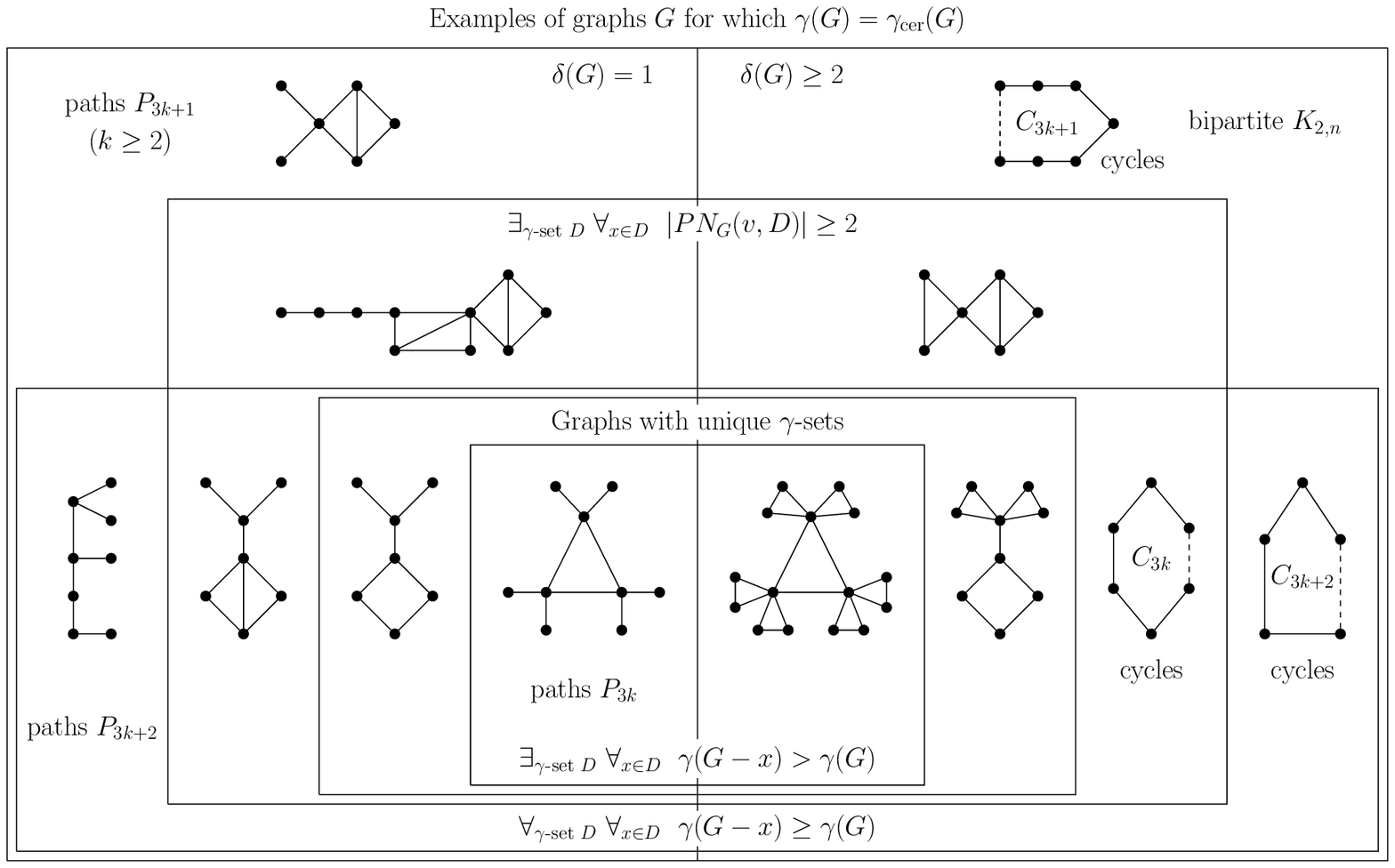}}
\caption{Examples of graphs $G$ for which $\gamma(G)=\gamma_\cer(G)$.} \label{rys-do-pracy}
\end{center}\end{figure}

A graph $G$ is said to be {\em $P_4$-free} if $P_4$ is not an induced subgraph of $G$.
We also say that a graph $G$ is $\gamma\gamma_{\cer}$-perfect if $\gamma(H) = \gamma_{\cer}(H)$ for each induced and connected subgraph $H\not=K_2$ of $G$. It follows from this definition that a graph $G$ is $\gamma\gamma_{\cer}$-perfect if and only if each non-2-element component of $G$ is $\gamma\gamma_{\cer}$-perfect. The path $P_4$ is the smallest non-$\gamma\gamma_{\cer}$-perfect graph. The union $ K_{2} \cup C_{4}$ is a~$\gamma\gamma_{\cer}$-perfect graph, while the union $ K_{2} \cup C_{5} $ is not a~$ \gamma\gamma_{\cer}$-perfect graph. We now study the equality $\gamma(G)=\gamma_{\cer}(G)$ for $P_4$-free graphs.

\begin{thm}\label{perfect} If $G$ is a connected $P_4$-free graph and $G \neq K_2$, then $\gamma(G) = \gamma_{cer}(G)$. \end{thm}

\begin{proof}
The result is obvious if $G=K_1$. Thus assume that $G$ is a connected $P_4$-free graph of order at least three. We shall prove that $\gamma(G) = \gamma_{\cer}(G)$. Since $\gamma(G) \leq \gamma_{\cer}(G)$, it suffices to show that some $\gamma$-set
of $G$ is a certified dominating set of $G$. If $X$ is a $\gamma$-set of $G$, then let
$p(X)$ denote the set $\{x\in X\colon |N_G(x) \cap (V_G - X)| = 1\}$. Now let $D$ be a $\gamma$-set of $G$. Since $G$ is connected and $G \neq K_2$, we may assume that $D$ contains no leaf of $G$. If $p(D)=\emptyset$, then $D$ is the required set. Thus assume that $p(D) \neq \emptyset$. But now what is left is to show that there exists another $\gamma$-set $D'$ of $G$ for which $|p(D')|<|p(D)|$.

Let $v$ be any vertex belonging to $p(D)$ and let $v'$ be the only neighbor of $v$ in $V_G-D$. Since $v$ is not a leaf and $v'$ is the only neighbor of $v$ in $V_G - D$, the set $N_G(v) \cap D$ is non-empty.  Choose any vertex $u \in N_G(v) \cap D$.
The minimality of $D$ and the fact that $v$ and $u$ are adjacent elements of $D$ imply  that the private neighborhoods $P\!N_G[v,D]$ and $P\!N_G[u,D]$ are disjoint and non-empty subsets of $V_G - D$. Certainly, $P\!N_G[v,D]= \{v'\}$ and, therefore, $N_G(x)\cap \{v'\}= \emptyset$ for every $x\in D-\{v\}$. Now, since $G$ is a $P_4$-free graph, the vertex $v'$ is not a leaf and it is adjacent to every vertex in $P\!N_G[u,D]$ and, therefore, to every vertex in $\bigcup_{u\in N_G(v)\cap D}P\!N_G[u,D]$. Again from the fact that $G$ is $P_4$-free it is easily seen that $N_G(v)\cap D=\{u\}$ and  $N_G(u)\cap D=\{v\}$. Now let us consider the set $D' = (D- \{v\}) \cup \{v'\}$. It is obvious that $D'$ is a~minimum dominating set of $G$ and $D'$ contains no leaf of $G$. It remains to show that $p(D')\subseteq  p(D)-\{v\}$. First, let us observe that $v \not\in p(D')$ (as $v\not\in D'$) and $\{v',u\}\cap p(D')= \emptyset$ (as $|N_G(x)\cap (V_G-D')|\ge |\{v\}
\cup P\!N_G[u,D]|\ge 2$ if $x\in \{v',u\}$). Now, since $N_G(v)\cap p(D')= \{v',u\}
\cap p(D')= \emptyset$, we have $N_G(x)\cap \{v\}= \emptyset$ for every $x\in p(D')$.
Consequently, if $x\in p(D')$, then $x\in p(D')-\{v, v', u\}$ and therefore we have
$1= |N_G(x)\cap (V_G-D')|= |N_G(x)\cap (V_G-((D- \{v\}) \cup \{v'\}))|=
|N_G(x)\cap (((V_G-D)- \{v'\}) \cup \{v\})|= |N_G(x)\cap ((V_G-D)- \{v'\}) \cup(N_G(x) \cap  \{v\})|= |((N_G(x)\cap (V_G-D))-(N_G(x)\cap \{v'\})) \cup(N_G(x) \cap  \{v\})|=
|N_G(x)\cap (V_G-D)|$, as $N_G(x) \cap  \{v'\}= \emptyset$ and $N_G(x) \cap  \{v\}= \emptyset$. This shows that $p(D')\subseteq  p(D)-\{v\}$ and completes the proof. \end{proof}

Theorem \ref{perfect} immediately implies the following characterization of the $\gamma\gamma_{\cer}$-perfect graphs.

\begin{wn} A graph $G$ is a $\gamma\gamma_{\cer}$-perfect graph if and only if $G$ is a $P_4$-free graph. \end{wn}

\begin{proof}     The ``only if'' part of the theorem follows from the fact that $ \gamma(P_{4}) = 2 < 4 = \gamma_{\cer}(P_4) $. The ``if'' part follows from Theorem \ref{perfect}. \end{proof}

The {\it corona} $H\circ K_1$ of a graph $H$ was defined in \cite{FruchtHarary} as the graph obtained from $H$ by adding a pendant edge to each vertex of $H$. A graph $G$ is said to be a {\em corona} if $G$ is the corona $H\circ K_1$ of some graph $H$. It is obvious that a corona is a graph in which each vertex is a leaf or it is adjacent to exactly one leaf. We now consider a generalization of this operation and study properties of certified dominating sets and certified domination numbers of resulting graphs.
In order that this presentation to be self-contained, we include the following definition. For a graph $G$ and a family ${\cal P}=\{\mathcal{P}(v)\colon v\in V_G\}$, where $\mathcal{P}(v)$  is a~partition of the neighborhood $N_G(v)$ of $v\in V_G$, the {\em $\mathcal{P}$-corona} of $G$, denoted by $G\circ \mathcal{P}$, is the graph defined in \cite{Dettlaff...Zylinski} and such that $$V_{G\circ \mathcal{P}}=\{(v,1)\colon v\in V_G\}\cup \bigcup_{v\in V_G} \{(v, A)\colon A\in \mathcal{P}(v)\}$$ and
$$E_{G\circ \mathcal{P}}= \bigcup_{v\in V_G} \{(v,1)(v,A)\colon A\in \mathcal{P}(v)\}\cup \bigcup_{uv\in E_G}\{(v,A)(u,B)\colon  (u\in A)  \wedge  (v\in B) \}.$$

\begin{example}{\rm
For the graph $G$ in Figure \ref{GGGGGoP} and a family ${\cal P}= \{\mathcal{P}(v) \colon v\in V_G\}$ of partitions of vertex neighborhoods of $G$, where ${\cal P}(v)= \{\{z\}, \{u\}\}$, ${\cal P}(u)=\{\{v,z\},\{w\}\}$, ${\cal P}(w)=\{\{u\}\}$, and ${\cal P}(z)=\{\{v,u\}\}$, the ${\cal P}$-corona of $G$ is the graph ${\cal P}\circ G$ shown on the right in  Figure \ref{GGGGGoP}.} \end{example}

\begin{figure}[h!] \begin{center} {\epsfxsize=4.5in \epsffile{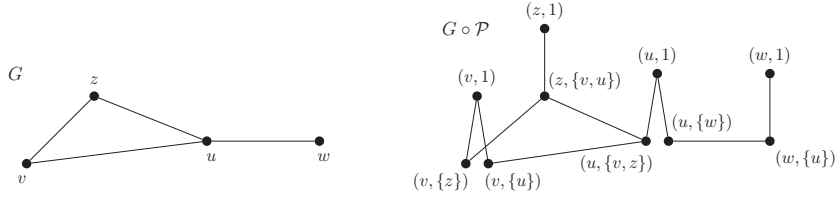}}%
\caption{Graphs $G$ and its ${\cal P}$-corona $G\circ {\cal P}$. } \label{GGGGGoP}\end{center}\end{figure}

In the next theorem we provide a necessary and sufficient condition for a $\mathcal{P}$-corona of $G$ to have $\gamma(G\circ {\cal P})= \gamma_{\rm cer}(G \circ {\cal P})$.  We start with the following two lemmas.

\begin{lem}\label{lem:gamma_eq_gammacer_no_leaves} Let $G$ be a connected graph of order at least three. If $D$ is a $\gamma_\cer$-set of $G$ and $D\cap L_G\not=\emptyset$, then $D$ is not a $\gamma$-set of $G$. Equivalently, if  $D$ is a $\gamma_\cer$-set of $G$ and $\gamma(G)=\gamma_{\cer}(G)$, then $L_G\cap D=\emptyset$ and $S_G\subseteq D$. \end{lem}

\begin{proof} Assume that $D$ is a $\gamma_\cer$-set of $G$ and $v\in D\cap L_G$. Then,  since $D$ is a certified dominating set of $G$, the only neighbor of $v$ is in~$D$. Therefore $D - \{v\}$ is a dominating set of $G$ and $\gamma(G)\le |D - \{v\}|<|D|$.
\end{proof}

\begin{lem}\label{lem:Pcoronas} If $G$ is a graph and ${\cal P}= \{\mathcal{P}(v)\colon v\in V_G\}$ is a family of partitions of vertex neighborhoods of $G$, then $\gamma(G\circ {\cal P})= |V_G|$. \end{lem}

\begin{proof} It is obvious from the definition of $G\circ \mathcal{P}$ that
$V_G\times \{1\}$ is a dominating set of $G\circ \mathcal{P}$, and therefore
$\gamma(G\circ \mathcal{P})\le |V_G\times \{1\}|=|V_G|$. On the other hand, let $D$ be a $\gamma$-set of $G\circ \mathcal{P}$. Then $D \cap N_{G\circ \mathcal{P}}[(v,1)]\not=
\emptyset$ for every $v\in V_G$, and, since the sets $N_{G\circ \mathcal{P}}[(v,1)]$ form a partition of $V_{G\circ \mathcal{P}}$, we have $\gamma(G\circ \mathcal{P})= |D|=
|\bigcup_{v\in V_G} \left(D\cap N_{G\circ \mathcal{P}}[(v,1)]\right)|= \sum_{v\in V_G}|D \cap N_{G\circ \mathcal{P}}[(v,1)]| \ge |V_G|$. Consequently,
$\gamma(G\circ {\cal P})= |V_G|$. \end{proof}

\begin{thm} \label{thm:Pcoronas}
If $G$ is a graph and ${\cal P}= \{\mathcal{P}(v)\colon v\in V_G\}$ is a family of partitions of vertex neighborhoods of $G$, then $\gamma(G\circ {\cal P})= \gamma_{\rm cer}(G \circ {\cal P})$ if and only if the set $\{u\in V_G \colon |{\cal P}(u)|\ge 2\}$ is a dominating set of $G$. \end{thm}

\begin{proof}
Assume first that ${\cal P}= \{\mathcal{P}(v)\colon v\in V_G\}$ is a family such
$D=\{u\in V_G \colon |{\cal P}(u)|\ge 2\}$ is a dominating set of $G$.
Consider the sets $\overline{D}=V_G-D=\{x\in V_G\colon |{\cal P}(x)|=1\}$, $D'=\{(x,N_G(x))\colon x\in \overline{D}\}$, $D''= D\times \{1\}$, and $\widetilde{D}=D'\cup D''$. Then $|\widetilde{D}|=|D'|+ |D''|=|V_G|$, and, since $|V_G|= \gamma(G\circ {\cal P})$ (by Lemma \ref{lem:Pcoronas}), $|\widetilde{D}|=\gamma(G\circ {\cal P})$. Moreover, since the sets $N_{G\circ {\cal P}}[(v,1)]$ form a~partition of $V_{G\circ {\cal P}}$, the set $D'$ dominates $\bigcup_{x\in \overline{D}}N_{G\circ {\cal P}}[(x,1)]$, while $D''$ dominates $\bigcup_{y\in D}N_{G\circ {\cal P}}[(y,1)]$, the set $\widetilde{D}$ is a minimum dominating set of $G\circ {\cal P}$. Therefore, taking into account Theorem~\ref{obs_gamma_rowne}, all we need is to prove that each vertex in $\widetilde{D}$ has at least two neighbors in $V_{G\circ {\cal P}}-\widetilde{D}$. This is obvious for vertices in $D''$, for if $(v,1)\in D''=D\times \{1\}$, then $N_{G\circ {\cal P}}((v,1))= \{(v,A)\colon A\in {\cal P}(v)\} \subseteq V_{G\circ {\cal P}}-\widetilde{D}$ and $|N_{G\circ {\cal P}}((v,1))|= |{\cal P}(v)|\ge 2$ (as $v\in D$). As regards a vertex $(v,N_G(v)) \in D'$, then $v\in \overline{D}$, and so $(v,1)$ is a neighbor of $(v,N_G(v))$ that is not an element of $\widetilde{D}$. Since $D$ is a dominating set of $G$ and $v\in \overline{D}$, the vertex $v$ is adjacent to some vertex $u$ in $D$. Let $A$ be the only set in ${\cal P}(u)$ that contains $v$. Then $(u,A)$ is another neighbor of $(v,N_G(v))$ that is not in $\widetilde{D}$. Consequently, each vertex in $\widetilde{D}$ has at least two neighbors in $V_{G\circ {\cal P}}-\widetilde{D}$, and hence $\gamma(G\circ {\cal P})= \gamma_{\rm cer}(G \circ {\cal P})$.

On the other hand, assume that $\gamma(G\circ {\cal P})= \gamma_{\rm cer}(G \circ {\cal P})$. Then, by Theorem \ref{obs_gamma_rowne}, $G\circ {\cal P}$ has a $\gamma$-set $\widetilde{D}$ such that $|N_{G\circ {\cal P}}(x)\cap (V_{G\circ {\cal P}}-\widetilde{D})|\ge 2$ for every $x\in \widetilde{D}$. We shall prove that $\{x\in V_G \colon |{\cal P}(x)| \ge 2\}$ is a dominating set of $G$. Let $v$ be a vertex of $G$ such that $|{\cal P}(v)|=1$. It suffices to show $|{\cal P}(u)|\ge 2$ for some vertex $u\in N_G(v)$. Suppose on the contrary that $|{\cal P}(u)|=1$ for every $u\in N_G(v)$.
Then $|{\cal P}(x)|=1$ and therefore ${\cal P}(x)=\{N_G(x)\}$ for every $x\in N_G[v]$.
Thus every vertex in $\{(x,1)\colon x\in N_G[v]\}$ is a leaf of $G\circ {\cal P}$ and $\{(x,1)\colon x\in N_G[v]\}\cap \widetilde{D}=\emptyset$ (by Lemma~\ref{lem:gamma_eq_gammacer_no_leaves}).
Hence the set of their only neighbors $\{(x,N_G(x))\colon x\in N_G[v]\}$ is a subset
of the dominating set $\widetilde{D}$ of $G\circ {\cal P}$. But now, since $N_{G\circ {\cal P}}((v,N_G(v)))- \{(v,1)\}= \{(x,N_G(x))\colon x\in N_G(v)\}\subseteq
\widetilde{D}$, we have $2\le |N_{G\circ {\cal P}}((v,N_G(v)))\cap (V_{G\circ {\cal P}}-\widetilde{D})|\le |\{(v,1)\}|=1$, which is impossible. This proves $\{x\in V_G\colon |{\cal P}(x)|\ge 2\}$ is a dominating set of $G$. \end{proof}

If $G$ is a graph and ${\cal P}=\{{\cal P}(v)\colon v\in V_G\}$ is a family of partitions of vertex neighborhoods of $G$ and ${\cal P}(v)=\{N_G(v)\}$ for every $v\in V_G$, then the resulting ${\cal P}$-corona $G\circ {\cal P}$ of $G$ is isomorphic to the corona  $G\circ K_1$ of $G$ and, in this case, $G\circ {\cal P}$ is said to be the corona of $G$. Since the family ${\cal P}=\{\{N_G(v)\}\colon v\in V_G\}$ does not satisfy the assumption of Theorem~\ref{thm:Pcoronas} and $G\circ \mathcal{P}=G\circ K_1$, we therefore have the following corollary.

\begin{wn} If $G$ is a graph, then  $\gamma(G\circ K_1)< \gamma_{\rm cer}(G\circ K_1)$.
\end{wn}

On the other hand, let $G$ be a graph and consider the family ${\cal P}= \{{\cal P}(v) \colon v \in V_G \}$ of partitions of vertex neighborhoods of $G$, where ${\cal P}(v)$ is the set of all $1$-element subsets of $N_G(v)$, that is, ${\cal P}(v)= \{\{x\} \colon x\in N_G(v)\}$, for every $v\in V_G$. Then the resulting ${\cal P}$-corona of $G$ is just the graph $S_2(G)$, called the {\em $2$-subdivision} of $G$ and obtained from $G$ by replacing each its edge $e = uv$ by a path $(u, u_e, v_e, v)$, where $u_e$ and $v_e$ are two new vertices. Now, if $G$ has no $2$-vertex component, then in the closed neighborhood of each vertex of $G$ there is a vertex of degree at least two, and therefore the family ${\cal P}$ satisfies the assumption of Theorem~\ref{thm:Pcoronas}. Thus, since $S_2(G)=G\circ \mathcal{P}$, we immediately have the following corollary.

\begin{wn} If $G$ is a graph with no $2$-vertex component, then $\gamma(S_2(G))= \gamma_{\rm cer}(S_2(G))$. \end{wn}

We finalize our discussion about the equality of the two domination parameters $\gamma$ and $\gamma_\cer$ for $\cP$-coronas of graphs by mentioning a hereditary property of the equality $\gamma=\gamma_{\rm cer}$. Let  $\mathcal{P}=\{\mathcal{P}(v)\colon v\in V_G\}$ and $\mathcal{P}'=\{\mathcal{P}'(v)\colon v\in V_G\}$ be two families of partitions of vertex neighborhoods of a graph $G$. We say that $\mathcal{P}'$ is a~{\em refinement} of $\mathcal{P}$, and write $\mathcal{P}'\prec \mathcal{P}$, if for every $v\in V_G$ and every $A\in \mathcal{P}'(v)$ there exists $B\in \mathcal{P}(v)$ such that $A\subseteq B$. Now it follows from Theorem~\ref{thm:Pcoronas} that if $\gamma(G\circ \mathcal{P})= \gamma_{\cer}(G\circ \mathcal{P})$,  then the equality $\gamma(G\circ \mathcal{P}')= \gamma_{\cer}(G\circ \mathcal{P}')$ also holds for any refinement $\mathcal{P}'$ of $\mathcal{P}$. It is then natural to ask which families ${\cal P}$ are extremal with respect to the equality $\gamma=\gamma_\cer$. We say that a family ${\cal P}$ of partitions of vertex neighborhoods of a graph $G$ is {\em maximal with respect to the equality $\gamma=\gamma_\cer$} if  $\gamma(G\circ \mathcal{P})= \gamma_{\cer}(G\circ \mathcal{P})$ and  $\gamma(G\circ \mathcal{P}'')< \gamma_{\cer}(G\circ \mathcal{P}'')$ whenever $\mathcal{P} \prec \mathcal{P}''$. We now have the following characterization of a maximal family of partitions of vertex neighborhoods of a graph.

\begin{thm}
Let ${\cal P}=\{\mathcal{P}(v)\colon v\in V_G\}$ be of a family of partitions of vertex neighborhoods of a graph $G$. Then the family ${\cal P}$  is maximal with respect to the equality $\gamma=\gamma_\cer$ if and only if $|\cP(v)|\le 2$ for each vertex $v$ of $G$, and the set $D=\{v \in V_G \colon |\cP(v)| = 2\}$ is a minimal dominating set of $G$.
\end{thm}

\begin{proof}
Assume that $\cP$  is maximal and suppose to the contrary that there is a vertex $v \in V_G$ such that  $|\cP(v)| > 2$, say $\cP(v)=\{X_1,\ldots, X_k\}$ and $k \ge 3$. Consider the family $\cP'=\{\mathcal{P}'(x)\colon x\in V_G\}$ in which $\cP'(v)= \{X_1 \cup \ldots \cup X_{k-1}, X_k\}$ and $\cP'(u)=\cP(u)$ if $u\in V_G-\{v\}$.
Then $\cP$ is a proper refinement of $\cP'$ and $\{x\in V_G\colon |{\cal P}(x)|\ge 2\}=
\{y\in V_G\colon |{\cal P}'(y)|\ge 2\}$. From this and from Theorem~\ref{thm:Pcoronas}
it follows that the equalities $\gamma(G\circ \cP)= \gamma_{\rm cer}(G \circ \cP)$ and
$\gamma(G\circ \cP')= \gamma_{\rm cer}(G \circ \cP')$ hold simultaneously, which contradicts the maximality of $\cP$. This proves that $|\cP(v)|\le 2$ for each vertex $v \in V_G$. Consequently, by Theorem~\ref{thm:Pcoronas}, $D= \{v \in V_G \colon |\cP(v)| = 2\}= \{v \in V_G \colon |\cP(v)| \ge 2\}$ is a dominating set of $G$. We now claim that $D$ is a minimal dominating set of $G$. Suppose not. Then $D-\{u\}$ is a dominating set of $G$ for some $u\in D$. Now $\cP$ is a proper refinement of the family $\cP''= \{\mathcal{P}''(y) \colon y\in V_G\}$ in which $\cP''(u)= \{N_G(u)\}$ and $\cP''(y)=\cP(y)$ for each $y\in V_G-\{u\}$. Consequently, since $\{y\in V_G \colon |{\cal P}''(y)|\ge 2\}= D-\{u\}$ is a dominating set of $G$, Theorem~\ref{thm:Pcoronas} implies that $\gamma(G\circ \cP'')= \gamma_{\rm cer}(G \circ \cP'')$, which again contradicts the maximality of $\cP$. This proves the minimality of $D$.

We now assume that ${\cal P}=\{\mathcal{P}(v)\colon v\in V_G\}$ is such a family of partitions of vertex neighborhoods of $G$ that $|\cP(v)|\le 2$ for each $v \in V_G$, and $D=\{v \in V_G \colon |\cP(v)| = 2\}$ is a~minimal dominating set of $G$. The last assumption and Theorem~\ref{thm:Pcoronas} imply that $\gamma(G\circ {\cal P})= \gamma_{\rm cer}(G \circ {\cal P})$. We now claim that the family $\cP$ is maximal.
Suppose to the contrary that $\cP$ is not maximal. Then $\cP$ is a proper refinement of some family $\cP'$ for which $\gamma(G\circ \cP')= \gamma_{\rm cer}(G \circ \cP')$.
The last equality and Theorem~\ref{thm:Pcoronas} guarantee that the set $D'=\{x \in V_G \colon |\cP'(x)| \ge 2\}$ is a dominating set of $G$. Finally, from the assumption and from the fact that ${\cal P}$ is a proper refinement of ${\cal P}'$, it follows that $2\ge |{\cal P}(x)|\ge |{\cal P}'(x)|\ge 1$ for every $x\in V_G$, $2\ge |{\cal P}(y)|> |{\cal P}'(y)|\ge 1$ for at least one $y\in V_G$, and therefore the dominating set $D'$ of $G$ is a proper subset of $D$, which was assumed to be a minimal dominating set of~$G$, a final contradiction. \end{proof}

\section{Properties of upper certified domination number}

In this section we study main properties of upper certified domination number $\Gamma_{\cer}$. We give a characterization of all graphs with $\Gamma_{\cer} = n$ and $\Gamma_{\cer} = n-2$, respectively. In addition, we focus on the relation between upper domination number and upper certified domination number of a graph. We start with the following useful lemma.

\begin{lem}\label{lem:C} Let $G$ be a connected graph of order at least two. If $D$ is a minimal certified dominating set of $G$ and $v$ is a~vertex such that $N_G[v]\subseteq D$, then  $v\in L_G\cup S_G^1$, that is, $v$ is a leaf or a weak support of $G$.  In addition, the induced subgraph $G[\{v\in D\colon N_G[v]\subseteq D\}]$ is a corona.  \end{lem}

\begin{proof} The result is obvious if $G=K_2$. Thus assume that $G$ is a graph of order at least three. Suppose, contrary to our claim, that the set $D'=\{v\in D\colon N_G[v]\subseteq D\}- (L_G\cup S_G^1)$ is non-empty. Let $F$ denote the subgraph $G[D']$. Let $I$ be a maximal set of independent vertices of degree at least two in $F$. The set  $I$ is a proper (possibly empty) subset of $D'$ and, if $I$ is non-empty, every vertex in $I$ has at least two neighbors in $D'-I$. Now let $I'$ denote the set $D'-N_F[I]$. We claim that $I'$     is dominated by the set $D-(L_G\cup D')$. This is trivially true if $I'=\emptyset$. Thus assume that $I'\not=\emptyset$ and $v_0\in I'$. Then $d_G(v_0)\ge 2$ (as $v_0\not \in L_G$) but $d_F(v_0)\le 1$ (otherwise $I\cup\{v_0\}$ would be a larger set of independent vertices of degree at least two in $F$) and therefore $N_G(v_0)-V_F$ is a~non-empty subset of $D-(L_G\cup D')$. This proves our claim. Consequently, $D-(L_G
\cup D')$ dominates $I'$ while $I$ is a certified dominating set of $F-I'$. This implies that $D-(D'-I)$ is a dominating set of $G$. In addition, it is no problem to observe that $D-((D'-I)\cup L')$ is a dominating set of $G$ for every subset $L'$ of $L_G$
(as every $x\in L_G$ is dominated by its only neighbor in $N_G(L_G)$ and $ N_G(L_G)\subseteq D$).

Let us consider the function $s\colon L_G\rightarrow S_G$, where $s(x)$ is the only neighbor of a leaf $x$, i.e. $s(x)$ is the only element of the set $N_G(x)$. This function is not necessarily an injection, but, since $D$ is a minimal certified dominating set of $G$, the restriction of $s$ to $L_G\cap D$ is indeed an injection and, in addition, the set $N_G[s(x)]$ is a~subset of $D$ for every $x\in L_G\cap D$. Moreover, the map $s\colon L_G\cap D\to S_G$ is a bijection between $L_G\cap D$ and $\{y\in S_G\colon N_G[y]\subseteq D\}$ $(=N_G(L_G\cap D))$. Let $S_0$ and $S_1$ be the sets $\{x\in N_G(L_G\cap D)\colon N_G(x)\cap (D'-I)= \emptyset\}$ and $\{x\in N_G(L_G\cap D)\colon N_G(x)\cap (D'-I)\not= \emptyset\}$, respectively. Now let $L_0 = N_G(S_0)\cap L_G$ $(=s^{-1}(S_0))$ and $L_1 = N_G(S_1)\cap L_G$ $(=s^{-1}(S_1))$. It is obvious that $S_0\cap S_1=\emptyset$, $S_0\cup S_1= N_G(L_G\cap D)$, $L_0\cap L_1= \emptyset$ and $L_0\cup L_1= L_G\cap D$.

We now prove that the set $D''= D-((D'-I)\cup L_1)$
is a certified dominating set of $G$. Since $D''$ is a dominating set of $G$, it suffices to show that no vertex belonging to $D''$ has exactly one neighbor in $V_G-D''$. To show this, let us take a vertex $x\in D''$. Since $D''\subseteq D$ and $D$ is a certified dominating set of $G$, we have $|N_G(x)\cap (V_G-D)|\ge 2$ or $|N_G(x)\cap (V_G-D)|=0$. In the first case $|N_G(x)\cap (V_G-D'')|\ge 2$, as $V_G-D \subseteq V_G-D''$. Thus assume that $|N_G(x)\cap (V_G-D)|=0$. Then $N_G[x]\subseteq D$ and, since $x\not\in (D'-I)\cup L_1$, $x$ is an element of the set $I\cup S_1 \cup (L_0\cup S_0)$. If $x\in I$, then $|N_F(x)\cap (D'-I)|\ge 2$ (by the definition of $I$) and therefore $|N_G(x)\cap (V_G-D'')|\ge 2$ (as $N_F(x)\subseteq N_G(x)$ and $D'-I\subseteq V_G-D''$). If $x\in S_1$ $(=N_G(L_1))$, then $N_G(x)\cap L_1\not=\emptyset$ and $N_G(x)\cap (D'-I)\not=\emptyset$, and therefore  $|N_G(x)\cap (V_G-D'')|\ge 2$ as $L_1$ and $D'-I$ are disjoint subsets of $V_G-D''$. It remains to show that $N_G(x)\cap (V_G-D'')=\emptyset$ (or, equivalently, that $N_G(x) \subseteq D''$) if $x\in L_0\cup S_0$. Since $D''= D-((D'-I)\cup L_1)$, it suffices to show that $N_G(x)\subseteq D$, $N_G(x)\cap (D'-I)=\emptyset$ and $N_G(x)\cap L_1= \emptyset$ if $x\in L_0\cup S_0$. We know already that $N_G[S_0]\subseteq D$, and hence $N_G(L_0\cup S_0)= N_G(L_0)\cup N_G(S_0)= S_0\cup N_G(S_0)= N_G[S_0]\subseteq D$, which proves that  $N_G(x)\subseteq D$ if $x\in L_0\cup S_0$. It follows from the definition of the set $S_0$ that $N_G(S_0)\cap (D'-I)=\emptyset$ and therefore  $N_G[S_0]\cap (D'-I)=\emptyset$ (as $S_0$ and $D'$ are disjoint). Consequently $N_G(L_0\cup S_0)\cap (D'-I)= N_G[S_0]\cap (D'-I)=\emptyset$ and this proves that $N_G(x)\cap (D'-I)=\emptyset$ if $x\in L_0\cup S_0$. Finally, it follows from the properties of the sets $S_0$, $S_1$, $L_0$ and $L_1$ that $N_G(L_1)\cap (L_0\cup S_0)= S_1\cap (L_0\cup S_0)= \emptyset$ and consequently $N_G(x)\cap L_1=\emptyset$ if $x\in L_0\cup S_0$. This completes the proof of the fact that $D''$ is a certified dominating set of $G$, which, however, contradicts the minimality of $D$ (as $D''$ is a~proper subset of $D$). This proves that the set $\{v\in D\colon N_G[v]\subseteq D\}$ is a subset of $L_G\cup S_G^1$, and therefore the induced subgraph  $G[\{v\in D\colon N_G[v]\subseteq D\}]$ is a corona.   \end{proof}

Now we give a characterization of the graphs for which the upper certified domination number is equal to their order.

\begin{thm} \label{thm-Gamma=n} Let $G$ be a graph of order $n$. Then the following statements are equivalent: \begin{enumerate} \item[$(1)$] every non-trivial component of $G$ is a corona; \item[$(2)$] $\gamma_{\cer}(G)=n$; \item[$(3)$] $\Gamma_{\cer}(G)=n$.
\end{enumerate} \end{thm}

\begin{proof} The equivalence of (1) and (2) has been proved in \cite{nasz}. It remains
to prove the equivalence of (2) and (3). If $\gamma_{\cer}(G)=n$, then $n= \gamma_{\cer}(G)\leq \Gamma_{\cer}(G)\leq n$ and, therefore,  $\Gamma_{\cer}(G)=n$. Assume now that $\Gamma_{\cer}(G)=n$. Then $V_G$ is a minimal certified dominating set of $G$, and, consequently, no proper subset of $V_G$ is a certified dominating set of $G$. This implies that it cannot be $\gamma_{\cer}(G)<n$, and, thus, it must be $\gamma_{\cer}(G) =n$. \end{proof}

It follows from the definition of a certified dominating set that if $G$ is a graph of order $n$, then no $n-1$ vertices form a certified dominating set of $G$. Consequently, either $\Gamma_{\cer}(G)=n$ or $\Gamma_{\cer}(G)\le n-2$. It is natural then to characterize all graphs $G$ of order $n$ for which $\Gamma_{\cer}(G)=n-2$. We need the following definitions. A {\em simple diadem\/} is a graph obtained from a corona by adding one new vertex and joining it to exactly one support of the corona, while a~{\em diadem\/} is a graph obtained  from a~corona by adding one new vertex and joining it to one leaf and its neighbor in the corona.

\begin{thm}\label{thm1} Let $G$ be a connected graph of order $n\geq 3$. Then $\Gamma_{\cer}(G)=n-2$ if and only if $G$ is a simple diadem, a diadem, or one of the   graphs $K_2 + \overline{K}_{n-2}$ and  $\overline{K}_2+\overline{K}_{n-2}$. \end{thm}

\begin{proof} It is a simple matter to observe that if a connected graph $G$ of order $n\ge 3$ is a simple diadem, a diadem, $K_2 + \overline{K}_{n-2}$ or $\overline{K}_2+\overline{K}_{n-2}$, then  $\Gamma_{\cer}(G) =n-2$.

Assume now that $G$ is a connected graph of order $n\geq 3$ for which $\Gamma_{\cer}(G) =n-2$. Let $D$ be a $\Gamma_{\cer}$-set of $G$, and let $x$ and $y$ be the only vertices in $V_G-D$. Let $D_0$ and $D_2$ be the sets $\{v\in D\colon N_G[v]\subseteq D\}$ and $\{v\in D\colon |N_G(v)\cap (V_G-D)|\geq 2\}$, respectively.

Assume first that $D_0=\emptyset$. Then $D_2=D=V_G-\{x,y\}$ and the minimality of $D$ easily implies the independence of $D$. Thus $G[D]= \overline{K}_{n-2}$ and now, from the fact that every vertex in $D$ is adjacent to both $x$ and $y$, it follows that $G$ is one of the graphs $K_2 + \overline{K}_{n-2}$ and $\overline{K}_2+\overline{K}_{n-2}$ (depending on whether or not $x$ and $y$ are adjacent).

Assume now that $D_0\not=\emptyset$. It follows from  Lemma~\ref{lem:C} that every vertex in $D_0$ is a leaf or a weak support of $G$ and the induced subgraph $G[D_0]$ of $G$ is a corona. We consider two subcases $|D_2|=1$ and $|D_2|\ge 2$ separately.

First assume that $|D_2|=1$ and let $z$ be the only vertex in $D_2$. Let $F$ be a component of $G[D_0]$. It follows from the connectivity of $G$ that $z$ is adjacent to at least one support vertex of $F$ if $F$ is of order at least four. If $F$ is of order two, then $z$ is adjacent to exactly one vertex of $F$ (for otherwise $z$ would be adjacent to both vertices of $F$ and then the proper subset $D-V_F$ of $D$ would be a certified dominating set of $G$, contrary to the minimality of $D$). From the above and from the fact that $z$ is adjacent to both $x$ and $y$ it follows that $G$ is a diadem or a simple diadem (depending on whether or not $x$ and $y$ are adjacent).

Finally assume that $|D_2|\ge 2$. We shall prove that this case is impossible.
The connectivity of $G$ implies that there exists a vertex $u\in D_2$ adjacent to same vertex in $D_0$. Let us consider the sets $L=\{t\in L_G\colon d_G(t,u)=2\}$ and $D'=D-
(\{u\}\cup L)$. Now we claim that the set $D'$ is a certified dominating set of $G$.
It is obvious that $L\subseteq L_G\subseteq D_0$ and therefore $D'= (D_2-\{u\})\cup (D_0-L)$ and $V_G-D'= \{x,y,u\}\cup L$. The vertices $x$ and $y$ are dominated by every vertex belonging to the non-empty set $D_2-\{u\}$. If $t\in L$, then its only neighbor is in $D_0-L$ and it dominates both $t$ and $u$. This proves that $D'$ is a dominating set of $G$. Thus it remains to observe that no vertex belonging to $D'$ has exactly one neighbor in $V_G-D'$. This is obvious for every vertex $t$ in  $D_2-\{u\}$, since $x$ and $y$ are two neighbors of $t$ in $V_G-D'$. Thus assume that $t$ is in $D_0-L$. Then, since $D_0-L= (S_G\cup L_G)-L= N_G(L)\cup N_G[L_G-L]$, either $t\in N_G(L)$ or $t\in  N_G[L_G-L]= (L_G-L)\cup N_G(L_G-L)$. If $t\in N_G(L)$, then $u$ and the only vertex in $N_G(t)\cap L$ are two neighbors of $t$ in $V_G-D'$. If $t\in  L_G-L$, then the only element of $N_G(t)$ belongs to $N_G(L_G-L)$ and so $N_G(t)\cap (V_G-D')=\emptyset$ (since $N_G(L_G-L)\subseteq N_G[L_G-L]\subseteq D'$). Finally, if $t\in N_G(L_G-L)$, then $t\in N_G(L_G)-N_G(L)$ and
$N_G(t)\cap (V_G-D')= \emptyset$ (because $V_G-D'= \{x,y,u\}\cup L$, $N_G(\{x,y,u\} \cup L)\subseteq \{x,y\}\cup D_2\cup N_G(L)$, and $t\not\in \{x,y\}\cup D_2\cup N_G(L)$).
This proves that the proper subset $D'=D- (\{u\}\cup L)$ of $D$ is a certified dominating set of $G$, contrary to the minimality of $D$. Therefore the case $|D_2|\ge 2$ is impossible, which completes the proof. \end{proof}

Next, we study the relation between upper domination number and upper certified domination number of a graph. The equality $\gamma(G)=\gamma_\cer(G)$ for any graph $G$ with $\delta(G)\ge 2$ (see Corollary \ref{theorem1-1}) could suggest that the analogous equality holds for the parameters $\Gamma$ and $\Gamma_\cer$.  Graph $G$ of Figure~\ref{fig:example_Gamma_vs_Gammacer} shows that this is not the case. For this graph we have $\delta(G)=2$, $\Gamma(G)=5$ and $\Gamma_\cer(G)=4$, and $\{v_2,v_3,v_6,v_7,v_9\}$ and $\{v_2,v_4,v_6,v_{10}\}$ are examples of $\Gamma$- and $\Gamma_{\cer}$-sets of $G$, respectively. For graphs $G$ with $\delta(G)\ge 2$, in fact, we always have $\Gamma_\cer(G)\le \Gamma(G)$.

\begin{figure}[h!] \begin{center} {\epsfxsize=4.5in \epsffile{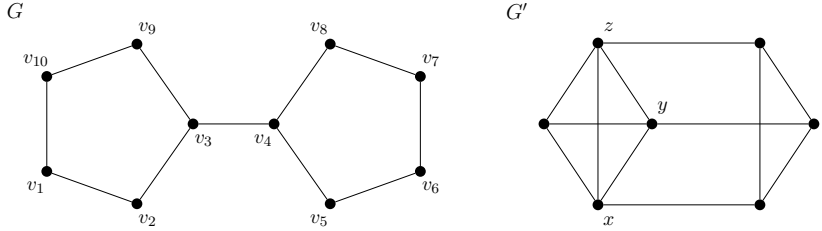}}
        \caption{Graphs $G$ and $G'$. } \label{fig:example_Gamma_vs_Gammacer}
        \label{fig:kontrprzyklad}\end{center}\end{figure}

\begin{lem}\label{lem:inequlity_Gamma_Gammacer} If $G$ is a graph with $\delta(G)\geq 2$, then $\Gamma_{\cer}(G)\leq \Gamma(G)$. \end{lem}

\begin{proof} Let $D$ be a $\Gamma_{\cer}$-set of $G$. Since $\delta(G)\geq 2$, it follows from Lemma~\ref{lem:C} that the set $\{v\in V_G\colon N_G[v]\subseteq D\}$ is empty. Hence $D$ is also a minimal dominating set of $G$ which implies that $\Gamma_{\cer}(G) \leq \Gamma(G)$. \end{proof}

Taking into account the above lemma, it is natural then to characterize all graphs $G$ with $\delta(G)\geq 2$ for which $\Gamma (G)=\Gamma_\cer (G)$. Here we have the following theorem.

\begin{thm}\label{thm:Gamma_Gammacer_leafless} Let $G$ be a connected graph with $\delta(G)\geq 2$. If $G$ has an independent $\Gamma$-set, then $\Gamma(G)= \Gamma_\cer(G)$. \end{thm}

\begin{proof} Let $D$ be an independent $\Gamma$-set of $G$. If $v\in D$, then $N_G(v)
\subseteq V_G-D$ (as $D$ is independent), $|N_G(v)|\ge 2$ (as $\delta(G) \ge 2$), and therefore $|N_G(v)\cap (V_G - D)| \ge 2$. This proves that $D$ is a minimal certified dominating set of $G$ and implies the inequality $\Gamma(G)  \le \Gamma_\cer(G)$. The last inequality and Lemma~\ref{lem:inequlity_Gamma_Gammacer} yield the equality
$\Gamma(G)= \Gamma_\cer(G)$. \end{proof}

As we have just seen, for graphs $G$ with $\delta(G)\geq 2$ the equality  $\Gamma(G)=\Gamma_\cer(G)$ holds, if $G$ has an independent $\Gamma$-set. It should be noted that the converse implication, however, is not true. For example, one can check that for the graph $G'$ shown in Figure \ref{fig:kontrprzyklad} is $\delta(G')\geq 2$ and  $\Gamma(G')= \Gamma_\cer(G')=3$, but the only $\Gamma$-set $\{x, y, z\}$ of $G'$ is not independent.

The {\em independence number} $\beta_0(G)$ of a graph $G$ is the cardinality of a largest independent set of vertices of $G$. It is well-known that $\beta_0(G) \le \Gamma(G)$  for every graph $G$. Therefore, by Theorem~\ref{thm:Gamma_Gammacer_leafless}, we immediately have our final corollary.

\begin{wn}\label{cor:Gamma_Gammacer_leafless_beta} If $G$ is a graph with $\delta(G)\geq 2$ and $\beta_0(G)=\Gamma(G)$, then also $\beta_0(G)=\Gamma(G)=\Gamma_\cer(G)$. \end{wn}

The equality of the parameters $\beta_0(G)$ and $\Gamma(G)$ has been studied by a number of authors (see for instance \cite{CHF94} and \cite[pp. 80-84]{HHS98}, and the references there) for well-known families of graphs, including strongly perfect graphs and their different subclasses: bipartite graphs, chordal graphs, and circular arc graphs, just to name a few. For all such graphs $G$, the equality $\Gamma(G)=\Gamma_\cer(G)$ is true if $\delta(G)\geq 2$.

\section{Closing open problems}
We close with the following list of open problems that we have yet to settle.

\begin{Problem} Determine the class of graphs $G$ for which $\gamma_{\rm cer}(G)=
\Gamma_{\rm cer}(G)$. \end{Problem}

\begin{Problem} Determine all the trees $T$ for which $\gamma_{\rm cer}(T)=
\gamma(T)$. \end{Problem}

\begin{Problem} Let $a, b, c, d$ be positive integers with $a\le b\le c \le d$.
Find necessary and sufficient conditions on  $a, b, c, d$ such that there exists
a graph $G$ with $\gamma(G)=a$, $\Gamma(G)=b$, $\gamma_{\rm cer}(G)=c$ and $\Gamma_{\rm cer}(G)=d$. Similarly, find necessary and sufficient conditions on  $a, b, c, d$ such that there exists a graph $G$ with $\gamma(G)=a$, $\gamma_{\rm cer}(G)=b$, $\Gamma(G)=c$ and $\Gamma_{\rm cer}(G)=d$. Finally, find necessary and sufficient conditions on  $a, b, c, d$ such that there exists a graph $G$ with $\gamma(G)=a$, $\gamma_{\rm cer}(G)=b$, $\Gamma_{\rm cer}(G)=c$ and $\Gamma(G)=d$. \end{Problem}


\begin{thebibliography}{99}

\bibitem{CHF94}
G.A.~Cheston and G. Fricke, Classes of graphs for which upper fractional domination equals independence, upper domination, and upper irredundance, {\em Discrete Appl. Math.} 55 (1994), 241--258.

\bibitem{Dettlaff...Zylinski}
M.~Dettlaff, M.~Lema\'{n}ska, J.~Topp, and P.~\.Zyli\'{n}ski,
Coronas and domination subdivision number of a graph,
{\em Bull. Malays. Math. Sci. Soc.} (2016). \texttt{doi:10.1007/s40840-016-0417-0}.

\bibitem{nasz} M.~Dettlaff, M.~Lema\'{n}ska, J.~Topp, R.~Ziemann, and P.~\.Zyli\'{n}ski,
Certified domination, submitted.

\bibitem{Fischermann} M.~Fischermann, Block graphs with unique minimum dominating sets, {\it Discrete Math.} 240 (2001), 247--251.

\bibitem{FruchtHarary} R. Frucht and F. Harary,
On the corona of two graphs, {\em Aequ. Math.} 4 (1970), 322--324.

\bibitem{GuntherHartnellMarkusRall} G.~Gunther, B.~Hartnell, L.R.~Markus, and D.~Rall,
Graphs with unique minimum dominating  sets, {\em Congr. Numer.} 101 (1994), 55--63.

\bibitem{HHS98} T.W.~Haynes, S.T.~Hedetniemi, and P.J.~Slater,
Fundamentals of Domination in Graphs, Marcel Dekker, New York, 1998.

\bibitem{HHS98b} T.W.~Haynes, S.T.~Hedetniemi, and P.J.~Slater,
Domination in Graphs: Advanced Topics, Marcel Dekker, New York, 1998.

\bibitem{HenningYeo} M.A. Henning and A. Yeo, Total Domination in Graphs, Springer-Verlag, New York, 2013.

\end{thebibliography}
\end{document}